\documentclass[10pt]{amsart}

\usepackage[english]{babel}
\usepackage{color}
\usepackage{
    mathtools, 
    amsmath, 
    amssymb, 
    enumerate, 
    bbm, 
    amsthm, 
    mathrsfs
}
\usepackage{graphicx}
\usepackage[utf8]{inputenc}
\usepackage[export]{adjustbox}
\usepackage{wrapfig}
\usepackage{hyperref}
\usepackage{fancyhdr}
\usepackage{xcolor}
\hypersetup{
    colorlinks=false,
    linkcolor=blue,
    filecolor=magenta,      
    urlcolor=cyan,
    pdftitle={CMU Math},
    pdfpagemode=FullScreen,
}

\newcommand{\R}{\mathbb{R}}

\newcommand{\N}{\mathbb{N}}
\newcommand{\E}{\mathbb{E}}
\newcommand{\PP}{\mathbb{P}}

\newcommand{\<}{\ensuremath{\langle}}
\renewcommand{\>}{\ensuremath{\rangle}}
\newcommand{\p}{\ensuremath{\partial}}

\newcommand{\abs}[1]{\left|#1\right|}

\newcommand{\paren}[1]{\left( #1 \right)}
\newcommand{\cparen}[1]{\left\{ #1 \right\}}
\newcommand{\bparen}[1]{\left[ #1 \right]}

\newcommand{\indi}[1]{\textbf{1}_{#1}}

\newcommand{\mean}[1]{\mathrm{mean}(#1)}

\newcommand{\measdott}[2]{\dot{#1}_{#2}}
\newcommand{\mudott}[1]{\measdott{\mu}{#1}}
\newcommand{\mudot}[1][]{\mudott{t}}
\newcommand{\independent}{\perp\!\!\!\!\perp} 

\newcommand{\law}[1]{\text{Law}\paren{#1}}

\newcommand{\divergence}[1]{\textup{div}\paren{#1}}
\newcommand{\trace}[1]{\mathrm{Tr}\paren{#1}}

\newcommand{\staticT}[2]{\overline{\mc{T}}^{#1, #2}}
\newcommand{\dynT}[2]{\mc{BB}^{#1, #2}}

\newcommand{\mc}{\mathcal}

\newcommand{\mr}{\mathrm}

\def\XXint#1#2#3{{\setbox0=\hbox{$#1{#2#3}{\int}$ }
\vcenter{\hbox{$#2#3$ }}\kern-.6\wd0}}

\makeatletter
\DeclareRobustCommand\widecheck[1]{{\mathpalette\@widecheck{#1}}}
\def\@widecheck#1#2{%
    \setbox\z@\hbox{\m@th$#1#2$}%
    \setbox\tw@\hbox{\m@th$#1%
       \widehat{%
          \vrule\@width\z@\@height\ht\z@
          \vrule\@height\z@\@width\wd\z@}$}%
    \dp\tw@-\ht\z@
    \@tempdima\ht\z@ \advance\@tempdima2\ht\tw@ \divide\@tempdima\thr@@
    \setbox\tw@\hbox{%
       \raise\@tempdima\hbox{\scalebox{1}[-1]{\lower\@tempdima\box
\tw@}}}%
    {\ooalign{\box\tw@ \cr \box\z@}}}
\makeatother

\newtheorem{thm}{Theorem}
\newtheorem{lemma}[thm]{Lemma}
\newtheorem{prop}[thm]{Proposition}
\newtheorem{cor}[thm]{Corollary}

\theoremstyle{definition}

\newtheorem{remark}[thm]{Remark}

\author{Ivan Guo}
\address{Ivan Guo\newline
\mbox{}\hspace{0.3cm} School of Mathematics \newline
\mbox{}\hspace{0.3cm} Monash University}
\email{Ivan.Guo@monash.edu}

\author{Severin Nilsson}
\address{Severin Nilsson \newline
\mbox{}\hspace{0.3cm} Department of Mathematics \newline
\mbox{}\hspace{0.3cm} Carnegie Mellon University}
\email{snilsson@andrew.cmu.edu}

\author{Johannes Wiesel}
\address{Johannes Wiesel \newline
\mbox{}\hspace{0.3cm} Department of Mathematics \newline
\mbox{}\hspace{0.3cm} University of Copenhagen}
\email{wiesel@math.ku.dk}

\title[Dynamic characterization of barycentric OT] {Dynamic characterization of barycentric optimal transport problems and their martingale relaxation}

\date{\today}

\begin{document}

\begin{abstract}
We extend the Benamou-Brenier formula from classical optimal transport to weak optimal transport and show that the barycentric optimal transport problem studied by Gozlan and Juillet has a dynamic analogue. We also investigate a martingale relaxation of this problem, and relate it to the martingale Benamou-Brenier formula of Backhoﬀ-Veraguas, Beiglb\"ock, Huesmann and K\"allblad.
\end{abstract}

\maketitle

\section{Introduction and main results}

Let $\mu$ and $\nu$ be two probability measures on $\R^d$ with finite second moments. The optimal transport problem with quadratic cost is given by
\begin{equation} \label{eq:OT}
    \mc{T}_2(\mu, \nu) 
    =
    \inf_{\pi \in \Pi(\mu, \nu)} \int |x-y|^2\,\pi(\mr{d}x, \mr{d}y),
    \tag{OT}
\end{equation}
where $\Pi(\mu, \nu)$ denotes the set of couplings between $\mu$ and $\nu$, i.e.,
\[
    \pi \in \Pi(\mu, \nu) 
    \iff
    \pi(A \times \R^d) = \mu(A) 
    ~ \text{and} ~ 
    \pi(\R^d \times A) = \nu(A) \quad \forall A \subseteq \R^d \text{ Borel;}
\]
see \cite{villani2021topics, santambrogio2015optimal} for an overview.
In the seminal work \cite{Benamou2000-yk} it is shown that solving $\mc{T}_2(\mu, \nu)$ is equivalent to minimizing the total energy along absolutely continuous curves $(\mu_t)_{t \in [0, 1]}$ from $\mu$ to $\nu$; to be precise,
\begin{equation} \label{eq:DOT}
    \mc{T}_2(\mu, \nu)
    =
    \inf_{(\mu_t, v_t)}
        \int_0^1 \int_{\R^d} \abs{v_t}^2 \mr{d}\mu_t \mr{d}t,
\end{equation}
where the infimum is taken over all $(\mu_t, v_t)$ such that $\mu_0 = \mu, \mu_1 = \nu$, and $(\mu_t, v_t)$ solves
\[
    \p_t \mu_t + \divergence{v_t \mu_t} = 0
\]
in the sense of distributions. Problem \eqref{eq:DOT} is known as the dynamic formulation of optimal transport, or the Benamou-Brenier formula. It has the probabilistic representation
\begin{align}\label{eq:DOT2}
    \mc{T}_2(\mu, \nu) 
    = 
    \inf \cparen{
    \E \bparen{ \int_0^1 \abs{v_t}^2 \mr{d}t }
    :\,
    \mr{d} X_t = v_t \mr{d}t
    ~ \text{where} ~ 
    X_0 \sim \mu, X_1 \sim \nu
    }.
    \tag{DOT}
\end{align}
In this note we extend the Benamou-Brenier formula to the so-called barycentric weak optimal transport problem. Introduced in the series of papers \cite{Gozlan2017-qo, Gozlan2018-fm}, this problem is defined as 
\begin{equation} \label{eq:WOT}
    \overline{\mc{T}}_2(\mu, \nu) 
    := 
    \inf_{\pi \in \Pi(\mu, \nu)} \int \abs{\mean{\pi_x} - x}^2 \,\mu(\mr{d} x),
    \tag{WOT}
\end{equation}
where the map $(\pi_x)_{x\in \R^d}$ is the disintegration of $\pi$ with respect to $\mu$ and 
$\mean{\rho} := \int y \,\rho(\mr{d} y)$ for any integrable probability measure $\rho$. Weak optimal transport covers the settings of martingale optimal transport \cite{beiglbock2013model, beiglbock2016problem}, entropic optimal transport \cite{conforti2019second, nutz2021introduction} and semi-martingale optimal transport \cite{tan2014optimal,guo2021path, benamou2024entropic}, among others;  see also the related works \cite{marton1996bounding, marton1996measure, talagrand1995concentration, talagrand1996new, fathi2018curvature, alibert2019new, bowles2018theory, fathi2018curvature, shu2020hopf} It has recently proved to be an extremely versatile tool in OT.
Intuitively, $\overline{\mc{T}}_2(\mu, \nu)$ measures how far $\mu$ and $\nu$ are away from being the marginals of a one-step martingale.  \cite{Gozlan_2020} show that 
\begin{align*}
\overline{\mc{T}}_2(\mu, \nu)=   \inf_{\eta\preceq_{c}\nu}  
\mc{T}_2(\mu, \eta),
\end{align*}
where $\preceq_c$ denotes convex order, i.e.~$\eta\preceq_{c}\nu$ if $\int f\mr{d}\eta\le \int f\mr{d}\nu$ for all convex functions $f:\R^d\to \R.$
Our first main result is the following dynamic characterization of 
$\overline{\mc{T}}_2$:

\begin{thm} \label{thm:intro1}
We have
\begin{align*}
\overline{\mc{T}}_2(\mu, \nu)
    =
    \inf \cparen{
    \E\bparen{ 
        \int_0^1 \abs{v_t}^2}:\,\mr{d}X_t = v_t \mr{d}t + \sigma_t \mr{d}B_t, ~
    X_0 \sim \mu, X_1 \sim \nu},
\end{align*}
where the infimum is taken over predictable processes $v$ and $\sigma$.
\end{thm}

Compared to \eqref{eq:DOT2}, the dynamic formulation in Theorem \ref{thm:intro1} allows for a costless martingale transport via the diffusion term $\sigma_t \mr{d}B_t$; on the flip side $\overline{\mc{T}}_2(\mu, \nu)$ penalizes only the deviation of $x\mapsto \mean{\pi_x}$ from the identity. 

We note that the dynamic formulation in Theorem \ref{thm:intro1} is different from the entropic projection problem, also known as the Schr\"odinger bridge,
\begin{align*}
\inf \cparen{
    \E\bparen{ 
        \int_0^1 \abs{v_t}^2}\,\mr{d}t:\,\mr{d}X_t = v_t \mr{d}t + \mr{d}B_t ~ \text{where} ~ 
    X_0 \sim \mu, X_1 \sim \nu},
\end{align*}
see \cite{schrodinger1932theorie, follmer2006random},
where the infimum is taken over the drift $v$ only and $\sigma$ is identically equal to the identity matrix. The Schr\"odinger bridge minimizes the Kullback-Leibler divergence of the law of $X$ with respect to the Wiener measure, rather than a cost function on the marginals.

As mentioned above, $\overline{\mc{T}}_2(\mu, \nu)$ essentially allows for arbitrary martingale transports, as $\sigma$ does not influence the cost $\E[\int_0^1  \abs{v_t}^2 dt]$. It is thus natural to extend our analysis to the functional
\begin{equation*}
\staticT{\alpha}{\beta}(\mu, \nu) 
    :=
    \inf_{\pi \in \Pi(\mu, \nu)} 
    \int 
        \alpha \abs{\mean{\pi_x} - x}^2 
        - 
        \beta \mr{MCov}(\pi_x, \gamma_1^d)
    \,\mu(\mr{d}x) 
\end{equation*}
for $\alpha, \beta> 0$, see \cite[Section 1.1.6]{beiglböck2025fundamentaltheoremweakoptimal} . In the above, the maximal covariance
\begin{align*}
    \mr{MCov}(\rho, \varrho)
    :=
    \sup_{\pi \in \Pi(\rho, \varrho)} 
    \int \langle y, z \rangle \,\pi(\mr{d}y, \mr{d}z), \quad \rho, \varrho \in \mathcal{P}_2(\R^d),
\end{align*}
measures the $2$-Wasserstein distance of the disintegration $\pi_x$ from the $d$-dimensional standard normal distribution $\gamma_1^d$, up to terms that do not depend on the coupling $\pi$. \\
One of the main results of \cite{backhoffveraguas2019martingalebenamoubrenierprobabilisticperspective} is the representation
\begin{align}\label{eq:kallblad}
\begin{split}
&\sup_{\pi \in \Pi_M(\mu, \nu)}
    \int
        \mr{MCov}(\pi_x, \gamma_1^d)
    \,\mu(\mr{d}x) \\
    &\qquad = \sup \cparen{
        \E \bparen{ \int_0^1 \trace{\sigma_t} \mr{d}t}
        :
        \mr{d}X_t = \sigma_t \mr{d}B_t, ~ X_0 \sim \mu, X_1 \sim \nu
    },
\end{split}
\end{align}
where 
\begin{align}\label{eq:MartMeas}
    \Pi_M(\mu, \nu)
    =
    \cparen{ \pi \in \Pi(\mu, \nu) : \mean{\pi_x} = x  \quad \forall x\in \R^d}
\end{align}
is the set of martingale measures with marginals $
\mu$ and $\nu$ and we recall that $\Pi_M(\mu,\nu)\neq \emptyset$ if and only if $\mu\preceq_c \nu$; see \cite{strassen1965existence}. The solution of \eqref{eq:kallblad} is given by a so-called stretched Brownian motion.
Equation \eqref{eq:kallblad} corresponds to $\staticT{0}{1}$ in our notation above. Our second main result result gives a similar representation of $\staticT{\alpha}{\beta}$ for the intermediate case $\alpha, \beta>0$.

\begin{thm} \label{thm:intro2}
For $\alpha, \beta > 0$ and $\mu, \nu \in \mc{P}_2(\R^d)$ we have 
\begin{align*}
    &\staticT{\alpha}{\beta}(\mu, \nu)\\
    &\quad =
    \inf \cparen{
    \E\bparen{ 
        \int_0^1 \alpha \abs{v_t}^2 
        - 
        \beta\paren{\<B_t, v_t \> + \trace{\sigma_t}} 
    \mr{d}t}: \,\mr{d}X_t = v_t \mr{d}t + \sigma \mr{d}B_t, ~
    X_0 \sim \mu, X_1 \sim \nu},
\end{align*}
where the infimum is taken over all predictable processes $v$ and $\sigma.$ The right hand side is attained by the process
\[
    \mr{d}X_t 
    = 
    (\nabla \varphi(X_0) - X_0) \mr{d}t + \sigma_t \mr{d}B_t
    \quad \text{with} \quad
    X_0 \sim \mu,
\]
where the $1$-Lipschitz map $\nabla \varphi$ is given in Proposition \ref{thm:weakOTThm} and $\sigma$ is given in Proposition \ref{prop:essential} below.
\end{thm}
Note that Theorem \ref{thm:intro1} can be formally obtained from Theorem \ref{thm:intro2} by taking $\alpha=1, \beta\to 0$; similarly \eqref{eq:kallblad} can be obtained by setting $\alpha\to \infty, \beta=1$. Let us also remark that one can actually restrict the minimization in Theorem \ref{thm:intro2} to drifts $v$ that are independent of $B$, leading to $\E[\langle B_t, v_t\rangle]=0$. This follows from the proof of Theorem \ref{thm:intro2} below. The dynamic formulation in Theorem \ref{thm:intro2} can also be seen as a version of the semimartingale optimal transport problem.

\section{Notation}

We write $\mc{P}_2(\R^d)$ for the set of (Borel) probability measures with finite second moments. We let $\< \cdot, \cdot \>$ denote the standard inner product on $\R^d$ and for $x \in \R^d$ we write $\abs{x}^2 = \<x, x \> $. For a probability measure $\mu$ on $\R^d$ and a function $\kappa:\R^d\to \mathcal{P}(\R^d)$ we define $(\mu \otimes \kappa_x)(A\times B) := \int_A \kappa_x(B)\,\mu(\mr{d}x)$ for all Borel sets $A,B \subseteq \R^d$. Next, we write $(\pi_x)_{x\in \R^d}$ for the disintegration of $\pi \in \Pi(\mu,\nu)$ wrt.~$\mu$, i.e.~$x\mapsto \pi_x(A)$ is Borel measurable for all Borel sets $A\subseteq \R^d$ and satisfies $\mu\otimes\pi_x=\pi.$ Lastly we define the push-forward measure of a function $f:\R^d\to \R^k$ under $\mu$ as $f_{\#}\mu(A):= \mu(\{x\in \R^d:\, f(x)\in A\})$ for all Borel sets $A\subseteq \R^k$, $k\in \N.$

We say that a process $X$ is an admissible diffusion process if there exists a filtered probability space $(\Omega, \mc{F}, (\mc{F}_t)_{t \in [0, 1]}, \PP)$ which supports a standard Brownian motion $(B_t)_{t \in [0, 1]}$ with $X_0 \independent (B_t)_{t\in [0,1]}$ and predictable processes $v \in L^2(\PP \otimes \mr{d}t; \R^d)$ and $\sigma \in L^2(\PP \otimes \mr{d}t; \R^{d\times d})$ such that
\[
    \mr{d}X_t = v_t \mr{d}t + \sigma_t \mr{d}B_t.
\]
For $\mu, \nu \in \mc{P}_2(\R^d)$, we denote by $\mc{D}(\mu, \nu)$ the set of all admissible diffusion processes $X$ with $X_0 \sim \mu$ and $X_1 \sim \nu$. We set $\gamma_t^d := \law{B_t}$. 
We also define 
\begin{equation*}
    \dynT{\alpha}{\beta}(\mu, \nu) 
    :=
    \inf_{X \in \mc{D}(\mu, \nu)}
    \E\bparen{ 
        \int_0^1 \alpha \abs{v_t}^2 
        - 
        \beta\paren{\<B_t, v_t \> + \trace{\sigma_t}} 
    \mr{d}t}.
\end{equation*}
Using this more compact notation, Theorem \ref{thm:intro1} reads
$
    \overline{\mc{T}}_2
    =
    \dynT{1}{0},
$
while Theorem \ref{thm:intro2} reads
$
    \staticT{\alpha}{\beta}
    =
    \dynT{\alpha}{\beta} 
$
for $\alpha, \beta > 0$.

\section{Preliminary results}

Before we turn to the proofs of Theorems \ref{thm:intro1} and \ref{thm:intro2}, we need to investigate the relation between two results, which were mentioned in the introduction.

\begin{prop}[{\cite[Theorem 2.2.]{backhoffveraguas2019martingalebenamoubrenierprobabilisticperspective}}] \label{thm:MBB}
Let $\mu, \nu \in \mc{P}_2(\R^d)$ with $\mu\preceq_c \nu$. Then 
\eqref{eq:kallblad} holds and the problem
\begin{equation*}
    \sup \cparen{
        \E \bparen{ \int_0^1 \trace{\sigma_t} \mr{d}t}
        :
        \mr{d}X_t = \sigma_t \mr{d}B_t, X_0 \sim \mu, X_1 \sim \nu
    }.
\end{equation*}
admits a unique (in law) maximizer $\widehat{M}$.
\end{prop}

The authors call the maximizer $\widehat{M}$ a \emph{stretched Brownian motion}; $\widehat{M}$ is the martingale $M$ whose trajectories are as close as possible to Brownian motion in the adapted Wasserstein distance, while satisfying the marginal conditions $M_0 \sim \mu$ and $M_1 \sim \nu$ (see \cite[Section 6]{backhoffveraguas2019martingalebenamoubrenierprobabilisticperspective}).

In the follow-up paper \cite{backhoffveraguas2025existencebassmartingalesmartingale} it is shown that under an irreducibility condition\footnote{Two measures $\mu$ and $\nu$ are irreducible if for any martingale $M$ with $M_0 \sim \mu$ and $M_1 \sim \nu$ we have the implication $\mu(A), \nu(B) > 0 \implies \PP(M_0 \in A, M_1 \in B) > 0$ for any $A, B \subseteq \R^d$ Borel.} on $\mu$ and $\nu$, $\widehat{M}$ is a \emph{Bass martingale} between $\mu$ and $\nu$. Bass martingales, which go back to \cite{bass1983skorokhod} as a solution to the Skorokhod embedding problem, are martingales $M$ of the form
\[
    M_t = \E \bparen{ \nabla \phi(W_1) | W_t },
\]
where the Brownian motion $W$ is started at some $W_0 \sim \alpha$, $\phi : \R^d \to \R$ is a convex function and $\nabla \phi(W_1)$ is square integrable. 
Bass' construction can be viewed as a natural analogue of Brenier's Theorem \cite{brenier1991polar}, which states that for regular enough measures $\mu$ and $\nu$, the minimizing vector field $v_t$ appearing in the dynamic formulation on $\mc{T}_2(\mu, \nu)$ is of the form $v_t = \nabla \phi-\text{Id}$ for some convex function $\phi$. 

Next we recall the following result of \cite{Gozlan_2020}, which was later refined in \cite{beiglböck2025fundamentaltheoremweakoptimal} and \cite{backhoffveraguas2025existencebassmartingalesmartingale}.

\begin{prop}[{\cite[Theorem 1.2]{Gozlan_2020}}] \label{thm:weakOTThm}
There exists a unique $\bar{\mu} \preceq_c \nu$ such
\[
    \overline{\mc{T}}_2(\mu, \nu)
    =
    \mc{T}_2(\mu, \bar{\mu}) 
    = 
    \inf_{\eta \preceq_c\nu} \mc{T}_2(\mu, \eta).
\]
In particular, $\bar{\mu}$ is given by 
\[
    \bar{\mu} = \nabla \varphi _\# \mu
\]
where $\varphi : \R^d \to \R$ is a convex $C^1(\R^d)$-function and $\nabla \varphi$ is $1$-Lipschitz. Furthermore, the optimizers of $\overline{\mc{T}}_2(\mu, \nu)$ and $\mc{T}_2(\mu, \bar{\mu})$ are connected via the relation
\begin{align*}
\begin{split}
    \pi \in \Pi(\mu, \nu) &\text{ is optimal for $\overline{\mc{T}}_2(\mu, \nu)$} 
    \\&\iff
    \pi_x = \kappa_{\nabla \varphi(x)} \text{ $\mu$-a.e for some } \kappa \in \Pi_M(\nabla \varphi _\# \mu, \nu),
\end{split}
\end{align*}
where $\Pi_M$ was defined in \eqref{eq:MartMeas}.
\end{prop}

We can now make a connection between Propositions \ref{thm:MBB} and \ref{thm:weakOTThm}: indeed, an admissible choice in Proposition \ref{thm:weakOTThm} is $\kappa = \text{Law}(\widehat{M_0}, \widehat{M}_1)$ where $\widehat{M}$ is a stretched Brownian motion between $\nabla\varphi_{\#}\mu$ and $\nu$ from Proposition \ref{thm:MBB}. In fact, the following holds:

\begin{prop}[{\cite[Theorem 5.4]{beiglböck2025fundamentaltheoremweakoptimal}}] \label{prop:essential}
Let $\varphi:\R^d\to \R$ be as in Proposition \ref{thm:weakOTThm} and let $\kappa=\text{Law}(\widehat{M}_0, \widehat{M}_1)$, where $\widehat{M}$ is a stretched Brownian motion between $\nabla\varphi _\# \mu$ and $\nu$. Then the coupling $\pi = \mu\otimes \kappa_{\nabla\varphi(x)} \in \Pi(\mu,\nu)$ is optimal for $\staticT{\alpha}{\beta}(\mu, \nu)$, for all $\alpha,\beta>0$.
\end{prop}

\section{Proofs}

We start with the following lemma.
\begin{lemma} \label{lem:1}
We have 
\[
    \dynT{1}{0}(\mu, \nu) 
    =
    \inf_{\eta \preceq_c \nu} \mc{T}_2(\mu, \eta).
\]

\begin{proof}
We begin by proving the inequality $\mc{T}_2(\mu, \eta) \geq \dynT{1}{0}(\mu, \nu)$ for any $\eta \preceq_c \nu$. Take any vector field $v \in L^2(\PP \otimes \mr{d}t; \R^d)$ that pushes $\mu$ onto $\eta$, i.e.
\[
    \mr{d}X_t = v_t \mr{d}t
    \quad \text{with} \quad
    X_0 \sim \mu, X_1 \sim \eta.
\]
Since $\eta \preceq_c \nu$, by the martingale representation theorem there exists $\sigma \in L^2(\PP \otimes \mr{d}t; \R^{d \times d}), M_0\independent (B_t)_{t\in [0,1]}$ such that
\begin{align}\label{eq:step1}
    \mr{d}M_t = \sigma_t \mr{d}B_t
    \quad \text{with} \quad 
    M_0 \sim \eta, M_1 \sim \nu.
\end{align}
For any $\varepsilon \in (0, 1)$ define the process $X^{\varepsilon}$ via
\begin{align}\label{eq:step2}
    \mr{d}X_t^{\varepsilon}
    =
    \frac{
        v_{\frac{t}{1 - \varepsilon}}
    }{
        1-\varepsilon
    }
    \indi{\cparen{0 \leq t \leq 1 - \varepsilon}}
    \mr{d}t
    +
    \frac{\sigma_{\frac{t + \varepsilon - 1}{\varepsilon}}
    }{
    \sqrt{\varepsilon}
    }
    \indi{\cparen{1 - \varepsilon < t \leq 1}}
    \mr{d}B_t
    \quad \text{with} \quad
    X_0^{\varepsilon} = X_0.
\end{align}
Then $X^{\varepsilon}$ is an element of $\mc{D}(\mu, \nu)$ and we have
\begin{align}\label{eq:step3}
    \dynT{1}{0}(\mu, \nu)
    \leq 
    \frac{1}{(1 - \varepsilon)^2}
    \E \bparen{ 
        \int_0^1 
            \big|v_{\frac{t}{1 - \varepsilon}}\big|^2 
            \indi{\cparen{0 \leq t \leq 1 - \varepsilon}} 
        \mr{d}t 
    }
    =
    \frac{1}{1 - \varepsilon}
    \E \bparen{ \int_0^1 \abs{v_t}^2 \mr{d}t }.
\end{align}
Minimizing over all such vector fields $v$, appealing to the Benamou-Brenier formula \eqref{eq:DOT2}, and taking $\varepsilon \downarrow 0$, we get the desired inequality $\dynT{1}{0}(\mu, \nu) \leq \mc{T}_2(\mu, \eta)$.

We now turn to proving the inequality $\inf_{\eta\preceq_c \nu}\mc{T}_2(\mu, \nu) \leq \dynT{1}{0}(\mu, \nu)$. Suppose that $X \in \mc{D}(\mu, \nu)$, i.e.~
\[
    \mr{d}X_t 
    = 
    v_t \mr{d}t + \sigma_t \mr{d}B_t
    \quad \text{with} \quad
    X_0 \sim \mu, X_1 \sim \nu.
\]
Let $Y$ be given by
\[
    \mr{d}Y_t = \E[ v_t | X_0 ] ~ \mr{d}t
    \quad \text{with} \quad
    Y_0 = X_0
\]
and set $\widehat{\mu} := \law{Y_1}$. Then $\widehat{\mu} \preceq_c \nu$ as 
\begin{align*}
    Y_1 
    &= X_0 + \int_0^1 \E[v_t | X_0 ] \mr{d}t
    =
    \E \bparen{
        X_0 + \int_0^1 v_t \mr{d}t
        \bigg| X_0
    }
    \\&=
    \E \bparen{
        X_0 
        + \int_0^1 v_t \mr{d}t
        + \int_0^1 \sigma_t  \mr{d}B_t
        ~ \bigg| X_0
    }
    =
    \E[X_1 | X_0].
\end{align*}
Thus, \eqref{eq:DOT2}, Jensen's inequality and Tonelli's theorem yield
\begin{align*}
    \inf_{\eta \preceq_c \nu }\mc{T}_2(\mu, \eta)
    \leq 
    \mc{T}_2(\mu, \widehat{\mu})
    &\leq 
    \E \bparen{ \int_0^1 \abs{\E[v_t | X_0] }^2 \mr{d}t } 
    \\&\leq  
    \E \bparen{ \int_0^1 \E[\abs{v_t}^2 | X_0] \mr{d}t } 
    =
    \E \bparen{ \int_0^1 \abs{v_t}^2 \mr{d}t }.
\end{align*}
As $X \in \mc{D}(\mu, \nu)$ was arbitrary, this concludes the proof.
\end{proof}
\end{lemma}

We now give the proof of Theorem \ref{thm:intro1}.
\begin{proof}[Proof of Theorem \ref{thm:intro1}]
We first show $\overline{\mc{T}}_2(\mu, \nu) \leq \dynT{1}{0}(\mu, \nu)$.
Take a process $X \in \mc{D}(\mu, \nu)$, i.e.
\[
    \mr{d}X_t 
    = 
    v_t \mr{d}t + \sigma_t \mr{d}B_t
    \quad \text{with} \quad
    X_0 \sim \mu, X_1 \sim \nu.
\]
By definition, $\text{Law}(X_0, X_1)\in \Pi(\mu,\nu)$. Applying Jensen's inequality,
\[
    \overline{\mc{T}}_2(\mu, \nu)
    \leq 
    \E \big[ \abs{ \E[X_1 | X_0] - X_0}^2 \big]
    =
    \E\bparen{ \abs{\E\bparen{ \int_0^1 v_t \mr{d}t\bigg| X_0}}^2 }
    \leq
    \E\bparen{ \int_0^1 \abs{v_t}^2 \mr{d}t }.
\]
Minimizing over $X$ yields the inequality $\overline{\mc{T}}_2(\mu, \nu) \leq \dynT{1}{0}(\mu, \nu)$. 

For the opposite inequality, let $(X_0, Y)\sim \pi\in \Pi(\mu,\nu)$. We set $v_t := \E[Y | X_0] - X_0$ and let $X$ solve $\mr{d} X_t=v_t \mr{d}t$. Note that here $v_t$ only depends on $X_0$ and is constant in $t$. Then $$\eta:=\text{Law}(X_1)=\text{Law}(\E[Y|X_0])\preceq_c \text{Law}(Y)=\nu.$$ We now define \eqref{eq:step1} and \eqref{eq:step2} as in the proof of Lemma \ref{lem:1} above to obtain
\[
    \dynT{1}{0}(\mu, \nu)
    \leq
    \E \bparen{ \int_0^1 \abs{v_t}^2 \mr{d}t }
    =
    \E \bparen{ \abs{ \E[Y | X_0] - X_0 }^2 }
\]
as in \eqref{eq:step3}.
Minimizing over $(X_0, Y)\sim \pi\in \Pi(\mu,\nu)$ concludes the proof.
\end{proof}

Combining Lemma \ref{lem:1} and the proof of Theorem \ref{thm:intro1} actually gives an independent proof of Proposition \ref{thm:weakOTThm}.
\begin{cor}
We have
\begin{align*}
     \overline{\mc{T}}_2(\mu, \nu) =\dynT{1}{0}(\mu, \nu) 
    =
    \inf_{\eta \preceq_c \nu} \mc{T}_2(\mu, \eta).
\end{align*}
\end{cor}

We now turn to the proof of Theorem \ref{thm:intro2}.

\begin{proof}[Proof of Theorem \ref{thm:intro2}]
Suppose that $X \in \mc{D}(\mu, \nu)$, i.e.~
\[
    \mr{d}X_t = v_t \mr{d}t + \sigma_t \mr{d}B_t
    \quad \text{with} \quad
    X_0 \sim \mu, X_1 \sim \nu,
\]
and define $\pi := \law{X_0, X_1}\in\Pi(\mu, \nu)$. Then
\begin{align}\label{eq:res1}
\begin{split}
    \int \abs{\mean{\pi_x} - x}^2 \,\mu(\mr{d} x)
    &=
    \E \bparen{
        \abs{ \E\bparen{ X_1 | X_0 } - X_0 }^2
    }
    \\&=
    \E \bparen{
        \abs{ \E \bparen{ 
            \int_0^1 v_t \mr{d}t 
            + 
            \int_0^1 \sigma_t  \mr{d}B_t 
            \bigg| X_0 
        }}^2
    }
    \\& =
    \E \bparen{
        \abs{ \E \bparen{ 
            \int_0^1 v_t \mr{d}t \bigg| X_0 
        }}^2
    }
    \leq 
    \E \bparen{ \int_0^1 \abs{v_t}^2 \mr{d}t },
\end{split}
\end{align}
where the last inequality follows by two applications of Jensen's inequality. Similarly, recalling that  $X_0 \independent (B_t)_{t\in [0,1]}$ and taking the possibly sub-optimal candidate $\varrho_x  := \law{X_1, B_1 | X_0 = x} \in \Pi(\pi_x, \gamma_1^d)$  yields
\begin{align}\label{eq:res2}
\begin{split}
    \int_{\R^d} 
        \mr{MCov}(\pi_x, \gamma_1^d)
    \mu(\mr{d} x)
    &\geq
    \E \bparen{ \E[\<X_1, B_1\> | X_0] }
    \\&=
    \E \bparen{ \< X_1, B_1 \> }
    =
    \E \bparen{ 
        \int_0^1 \<v_t, B_t \> + \trace{\sigma_t} \mr{d}t 
    }.
\end{split}
\end{align}
Combining \eqref{eq:res1} and \eqref{eq:res2} we deduce the inequality
\begin{align*}
    \int_{\R^d} 
        \alpha \abs{\mean{\pi_x} - x}
        -
        &\beta \mr{MCov}(\pi_x, \gamma_1^d)
     \mu(\mr{d}x)
    \\&\leq 
    \E \bparen{ 
        \int_0^1 
            \alpha \abs{v_t}^2 
            - 
            \beta \paren{\< v_t, B_t \> + \trace{\sigma_t}}
        \mr{d}t
    },
\end{align*}
showing $\staticT{\alpha}{\beta}(\mu, \nu) \le \dynT{\alpha}{\beta}(\mu, \nu)$.

For the inequality $\staticT{\alpha}{\beta}(\mu, \nu) \geq \dynT{\alpha}{\beta}(\mu, \nu)$, let $\kappa$ and $\nabla \varphi$ be as in Proposition \ref{thm:MBB} and \ref{thm:weakOTThm}, i.e.~$\kappa = \text{Law}(\widehat{M}_0, \widehat{M}_1)$ where $\widehat{M}$ denotes the stretched Brownian motion from $\nabla \varphi_\#\mu$ to $\nu$. Let us take $X_0 \sim \mu$ and apply the martingale representation theorem to write
\[
    \widehat{M}_t = 
    \nabla \varphi(X_0) + \int_0^t \sigma_s \mr{d}B_s 
\]
for some  $\sigma \in L^2(\PP \otimes \mr{d}t; \R^{d\times d})$ and $X_0\independent (B_t)_{t\in [0,1]}$.
Next, we set $v_t = \nabla \varphi(X_0) - X_0$ and define the process $X$ via
\begin{align*}
    \mr{d}X_t 
    &= 
    v_t \mr{d}t + \sigma_t \mr{d}B_t.
\end{align*}
By definition, $\pi := \law{X_0, X_1}$ is an element of $\Pi(\mu, \nu)$ and $\pi_x=\kappa_{\nabla \varphi(x)}$. By Proposition \ref{prop:essential} we conclude that $\pi$ is the minimizer of $\staticT{\alpha}{\beta}(\mu, \nu)$. Furthermore,
\begin{align}\label{eq:goal1}
    \E \bparen{
        \int_0^1 |v_t|^2 \mr{d}t
    }
    =
    \E \bparen{ \abs{ \nabla \varphi(X_0) - X_0 }^2 }
    =
    \int \abs{\mean{\kappa_{\nabla \varphi(x)}} - x}^2 
    \,\mu(\mr{d}x).
\end{align}
Next we observe that by Proposition \ref{thm:MBB},
\begin{align}\label{eq:goal2}
    \int
        \mathrm{MCov}(\kappa_{\nabla \varphi (x)}, \gamma_1^d)
    \,\mu(\mr{d}x)= \int
        \mathrm{MCov}(\pi_x, \gamma_1^d)
    \,\mu(\mr{d}x)=  \E \bparen{ \int_0^1 \trace{\sigma_t} \mr{d}t }.
\end{align}
Lastly, by Fubini's theorem and $X_0 \independent (B_t)_{t\in [0,1]} $, we have 
\begin{align}\label{eq:goal3}
\begin{split}
    \E \bparen{
        \int_0^1 \< v_t, B_t \> \mr{d}t 
    }
    &=
    \int_0^1 \E \bparen{ \< \nabla \varphi(X_0) - X_0 , B_t \> } \mr{d}t
    \\&=
    \int_0^1 \< \E \bparen{\nabla \varphi(X_0) - X_0} , \E[B_t] \> \mr{d}t
    =
    0.
\end{split}
\end{align}
Combining \eqref{eq:goal1}-\eqref{eq:goal3} and using optimality of $\pi$ we obtain
\begin{align*}
    \staticT{\alpha}{\beta}(\mu, \nu)
    &=
    \int_{\R^d} 
        \alpha \big|x - \mean{\kappa_{\nabla \varphi(x)}}| 
        - 
        \beta \mr{MCov}(\kappa_{\nabla\varphi (x)}, \gamma_1^d)
    \,\mu(\mr{d} x)
    \\&=
    \E \bparen{ 
        \int_0^1 
            \alpha \abs{v_t}^2 
            - 
            \beta \paren{\< v_t, B_t\> + \trace{\sigma_t}}
        \mr{d}t
    }
    \geq 
    \dynT{\alpha}{\beta}(\mu, \nu).
\end{align*}
This concludes the proof.
\end{proof}

\begin{remark}
    In Theorems \ref{thm:intro1} and \ref{thm:intro2}, the quadratic cost function can be generalized to any convex cost function using the same argument, noting that {\cite[Theorem 5.4]{beiglböck2025fundamentaltheoremweakoptimal}} also holds for general convex cost functions. This is analogous to the extension of the Benamou-Brenier formula to convex cost functions \cite{brenier2004extended, pass2025dynamical}.
\end{remark}

\bibliographystyle{amsalpha}
\bibliography{bib}

\end{document}